\documentclass[11pt, reqno]{amsart}
\usepackage{hyperref}
\usepackage{cmap}                        
\usepackage[cp1251]{inputenc}            
\usepackage[english]{babel}
\usepackage[left=1in,right=1in,top=1in,bottom=1.5in]{geometry} 
\usepackage{amssymb,amsmath, amsthm, amscd,ifthen}
\usepackage{graphicx}

\newtheorem*{thm*}{Theorem}
\newcommand{\ff}{{\mathcal F}}
\newcommand{\aaa}{{\mathcal A}}

\newcommand{\G}{{\mathcal G}}
\newcommand{\h}{{\mathcal H}}
\newcommand{\T}{{\mathcal T}}

\newtheorem*{cla*}{Claim}
\newcommand{\bb}{{\mathcal B}}
\newcommand{\ccc}{{\mathcal C}}

\newtheorem{thm}{Theorem}
\newtheorem{gypo}{Conjecture}

\newtheorem{cla}[thm]{Claim}

\newtheorem{cor}[thm]{Corollary}
\date{}

\newtheorem{prop}[thm]{Proposition}
\newtheorem{obs}[thm]{Observation}

\newtheorem{defn}[thm]{Definition}

\date{}
\title{Incompatible intersection properties}
\author{Peter Frankl}\address{R\'enyi Institute, Budapest, Hungary; Email: {\tt peter.frankl@gmail.com}}

\author{Andrey Kupavskii}
\address{University of Birmingham,
Moscow Institute of Physics and Technology and Caucasus Mathematical Center, Adyghe State University; Email: {\tt kupavskii@ya.ru}.} \thanks{The research of the second author was partially supported by the  EPSRC grant no. EP/N019504/1.}

\begin{document}
\maketitle
\begin{abstract} Let $\ff\subset 2^{[n]}$ be a family in which any three sets have non-empty intersection and any two sets have at least $38$ elements in common. The nearly best possible bound $|\ff|\le 2^{n-2}$ is proved. We believe that $38$ can be replaced by $3$ and provide a simple-looking conjecture that would imply this.\end{abstract}
\section{Introduction}
Let $[n]:=\{1,\ldots, n\}$ be the standard $n$-element set and $2^{[n]}$ its power set. Subsets of $2^{[n]}$ are called {\it families}.
\begin{defn}
  For positive integers $r,t$, where $r\ge 2$, a family $\ff\subset 2^{[n]}$ is called {\it $r$-wise $t$-intersecting}, if $|F_1\cap \ldots \cap F_r|\ge t$ for all $F_1,\ldots, F_r\in\ff$.
\end{defn}
In the case $t=1$, instead of $1$-intersecting the term {\it intersecting} is used.
Arguably the simplest result in extremal set theory is the following.

\begin{prop}
  If $\ff\subset 2^{[n]}$ is $2$-wise intersecting then \begin{equation}\label{eqnew1}
                                                          |\ff|\le 2^{n-1}.
                                                        \end{equation}
\end{prop}
The above result is a small part of the classical Erd\H os-Ko-Rado paper \cite{EKR}. Since the family $\ff_0:=\{F\subset 2^{[n]}: 1\in F\}$ is $r$-wise intersecting for every $r\ge 2$, \eqref{eqnew1} is the best possible bound for $r\ge 3$ as well. The family $\ff_0$ is usually called trivially intersecting.

Let us call a family {\it non-trivial} if $\bigcap_{F\in \ff} F=\emptyset$.
The following result is one of the early gems in extremal set theory.
\begin{thm}[Brace-Daykin \cite{BD}] Suppose that $\ff\subset 2^{[n]}$ is $r$-wise intersecting and non-trivial. Then \begin{equation}\label{eqnew2}
                    |\ff|\le \frac{r+2}{2^r}2^{n-1}.
                  \end{equation}
\end{thm}
Since $r+2<2^r$ for $r\ge 3$ and $(r+2)2^{-r}\to 0$ as $r$ tends to infinity, \eqref{eqnew2} is much stronger than \eqref{eqnew1}. The following example shows that it is best possible for $n\ge r+1$.
$$\bb(1,r):=\{B\subset [n]:|B\cap [r+1]\ge r\}.$$
Let us mention that for $n\le r$ there is no non-trivial $r$-wise intersecting family. For a simple proof of \eqref{eqnew2} cf. \cite{Fra3}.

\begin{defn}
  For a family $\ff\subset 2^{[n]}$ and an arbitrary integer $r\ge 2$ let $t(\ff,r)$ denote the largest integer $t$ such that $|F_1\cap\ldots \cap F_r|\ge t$ for all $F_1,\ldots, F_r\in \ff$.
\end{defn}
One can easily check that $t(\ff,r+1)\le \max\{0,t(\ff,r)-1\}$ for non-trivial families. Therefore, $t(\ff,2)\ge 2$ for every non-trivial $3$-wise intersecting family $\ff$. On the other hand, we believe that assuming $t(\ff,2)\ge 3$ leads to stronger bounds on the size of the family.

\begin{gypo}\label{conj3}
  Suppose that $\ff\subset 2^{[n]}$ is both $3$-wise $1$-intersecting and $2$-wise $3$-intersecting. Then \begin{equation}\label{eqnew3}
         |\ff|\le 2^{n-2}.
       \end{equation}
\end{gypo}
If $\ff$ is trivial, e.g., if $1\in F$ for all $F\in \ff$, then the $2$-wise $3$-intersecting property implies that $\ff(1):=\{F\setminus \{1\}: F\in\ff\}\subset 2^{[2,n]}$ is $2$-wise intersecting. Applying \eqref{eqnew1} to $\ff(1)$ yields
$$|\ff|=|\ff(1)|\le 2^{n-2}.$$
This shows that in proving \eqref{eqnew3} one might assume that $\ff$ is non-trivial. From \eqref{eqnew2} we obtain $|\ff|\le \frac 5{16}2^n = \frac 54 \cdot 2^{n-2}$, which falls short of \eqref{eqnew3}.

\textbf{Example. } Let $t\ge 2$ be a fixed integer and suppose for convenience that $n>t$, $n+t$ is odd. Define
$$\T(n,t):=\Big\{\{1\}\cup T: T\subset [2,n], |T|\ge \frac {n-1+t}2\Big\}.$$
\begin{cla} The following hold:
\begin{itemize}
  \item[(i)] $\T(n,t)$ is $3$-wise intersecting and $2$-wise $(t+1)$-intersecting.
  \item[(ii)] $|\T(n,t)| = \sum_{i\ge \frac{n-1+t}2}{n-1\choose i} = (1-o(1))2^{n-2}$ as $n\to \infty$.
\end{itemize}
\end{cla}
We leave the easy proof to the reader.
This claim shows that even for $t$ large one cannot expect something much smaller than $2^{n-2}$.

We were unable to prove Conjecture~\ref{conj3}, but established \eqref{eqnew3} with $3$ replaced by $38$.
\begin{thm}~\label{thmnew1}
  Suppose that $\ff\subset 2^{[n]}$ is $3$-wise intersecting and $2$-wise $38$-intersecting. Then \eqref{eqnew3} holds.
\end{thm}
A family $\ff\subset 2^{[n]}$ is called an {\it up-set} if for all $F\in\ff$, $F\subset H\subset [n]$ implies $H\in \ff$.
Every family generates a unique up-set containing it. Moreover, if it is $r$-wise $t$-intersecting then the same holds for the corresponding up-set. Therefore, unless otherwise stated, we shall tacitly assume that the families we consider are up-sets.

Let us mention that the Katona Theorem~\cite{Kat64} determines the maximum size $k(n,t)$ of $2$-wise $t$-intersecting families for all $n\ge t\ge1$. The construction is analogous to $\T(n,t)$ and shows $$k(n,t)=(1-o(1))2^{n-1}\text{ for }t\text{ fixed and }n\to \infty.$$
That is, for each of the two intersecting properties from Theorem~\ref{thmnew1}, we have a lower bound of the form $(1+o(1))2^{n-1}$ for the largest size of the family satisfying the property. By the lemma of Kleitman \cite{Kl66}, two up-sets $\ff_1,\ff_2\subset 2^{n}$ of sizes $2^{n-\alpha_1},2^{n-\alpha_2}$, respectively, satisfy $|\ff_1\cap \ff_2|\ge2^{n-\alpha_1-\alpha_2}$. This immediately gives us a lower bound of $(1+o(1))2^{n-2}$ for the largest size of the family satisfying the conditions of Theorem~\ref{thmnew1}. Thus, one may say that, in a sense, $3$-wise intersecting and $2$-wise $t$-intersecting properties are as incompatible for large families as any two monotone increasing properties may be.

For a family $\ff$, let $\partial (\ff)$ be its {\it immediate shadow}:
$$\partial \ff:=\{G:\exists F\in \ff, G\subset F, |F\setminus G|=1\}.$$
Define also $\sigma(\ff):=\ff\cup \partial \ff$.

It is important to note that $[n]\in \ff$ for every non-empty up-set $\ff\subset 2^{[n]}$. This implies ${[n]\choose n-1}\subset \partial\ff$ whence both $\partial \ff$ and $\sigma(\ff)$ are non-trivial.

\begin{gypo}\label{conj9}
Suppose that $\ff\subset 2^{[n]}$ is $3$-wise intersecting. Then
\begin{equation}\label{eqnew4}
  |\sigma(\ff)|\ge2|\ff|.
\end{equation}
\end{gypo}
In the next section we show that Conjecture~\ref{conj9} implies Conjecture~\ref{conj3}.

\section{Preliminaries}

There is a natural partial order $A\prec B$ defined for sets of the same size. Suppose that $A = \{a_1,\ldots, a_p\}$, $B=\{b_1,\ldots,b_p\}$ are distinct sets with $a_1<\ldots <a_p$ and $b_1<\ldots <b_p$. We write $A\prec B$ iff $a_i\le b_i$ for all $1\le i\le p$.
\begin{defn}
  The family $\ff\subset 2^{[n]}$ is called {\it initial} if $A\prec B$ and $B\in \ff$ imply $A\in \ff$.
\end{defn}
Extend the above partial order to $2^{[n]}$ by putting $A\prec B$ if $B\subset A$. We call this order the {\it shifting/inclusion order}.
Erd\H os, Ko and Rado \cite{EKR} defined an operation on families of sets (called {\it shifting}) that maintains the $r$-wise $t$-intersecting property (cf. \cite{Fra3} for the proof). Since repeated application of shifting always produces an initial family, we shall always assume that the families in question are initial.

\begin{prop}[\cite{F11}]\label{prop11}
  If $\ff\subset 2^{[n]}$ is $3$-wise $t$-intersecting and initial, then, for every $F\in \ff$, there exists an integer $\ell\ge 0$  such that \begin{equation}\label{eqnew5}
                                   |F\cap [3\ell+t]|\ge 2\ell+t.
                                 \end{equation}
\end{prop}
The following result is proven in \cite{F5}.
\begin{thm}[\cite{F5}]\label{thm12}
  Suppose that $\ff\subset 2^{[n]}$ is such that for any $F\in \ff$ we have $|F\cap [3\ell+2]|\ge 2\ell+2$ for some $\ell\ge 0$. Then
  \begin{equation}\label{eqnew6} |\partial(\ff)|\ge 2|\ff|.
  \end{equation}
\end{thm}
\begin{cor}\label{cor10}
  Suppose that $\emptyset\ne \ff\subset 2^{[n]}$ is $3$-wise $2$-intersecting. Then $\sigma(\ff)>2|\ff|$.
\end{cor}
\begin{proof}
  Proposition~\ref{prop11} implies that $\ff$ satisfies the conditions of Theorem~\ref{thm12}. Now the statement follows from \eqref{eqnew6} and $[n]\notin \partial \ff$.
\end{proof}
\begin{defn}
  Suppose that $\aaa,\bb,\ccc\subset 2^{[n]}$ satisfy $|A\cap B\cap C|\ge t$ for all $A\in \aaa, B\in\bb$ and $C\in\ccc$. Then we say that $\aaa, \bb,\ccc$ are {\it cross-$t$-intersecting}.
\end{defn}
Let us recall the following recent result.
\begin{thm}[\cite{F10}]
  Suppose that $\aaa,\bb,\ccc\subset 2^{[n]}$ are non-trivial and cross-$1$-intersecting. Then
  \begin{equation}\label{eqnew7} |\aaa|+|\bb|+|\ccc|<2^n.
  \end{equation}
\end{thm}
The reason for our interest in $\partial \ff$ and $\sigma(\ff)$ is explained by the following simple statement.
\begin{obs}
  If $\aaa,\bb,\ccc\subset 2^{[n]}$ are cross-$t$-intersecting, $t\ge 2$, then $\sigma(\aaa), \bb,\ccc$ are cross-$(t-1)$-intersecting.
\end{obs}

We finish this section with a short proof of the fact that Conjecture~\ref{conj9} implies Conjecture~\ref{conj3}.
\begin{proof}[Conjecture~\ref{conj9} implies Conjecture~\ref{conj3}] Consider $\ff$ as in the statement of Conjecture~\ref{conj3}. Then $\sigma(\ff)$ is $2$-wise intersecting, and thus $|\sigma(\ff)|\le 2^{n-1}$. Therefore, by Conjecture~\ref{conj9}, $|\ff|\le \frac 12 |\sigma(\ff)|\le 2^{n-2}$.
\end{proof}

\section{Proof of Theorem~\ref{thmnew1}}
Consider a shifted family $\ff\subset 2^{[n]}$ as in the statement of Theorem~\ref{thmnew1}. For $S\subset [s]$, define
$$\ff(S,[s]):=\{F\setminus S: F\in\ff, F\cap [s]=S\}.$$

We consider two cases depending on whether the subsets not containing $1$ have a strong or weak presence in $\ff$. As a criterion, let us fix the set
$$H_0:=\{[2,8]\cup \{10,11,13,14,16,17\ldots\}\cup [n].$$
Note that for all $t$, $3\le t\le n/3$,
\begin{equation}\label{eq8}
  |H_0\cap [3t]|=2t+1.
\end{equation}
\textbf{Case 1. $H_0\in\ff$. } Put $w:=54$. We are going to partition $\ff$ according to $F\cap [w]$.
Set $\tilde H:=H_0\cap [w]$ and define
$$\mathcal G_0:=\{G\subset [w]:|G|\ge 33, G\ne \tilde H\}.$$
It is easy to verify by computer-aided computation that
\begin{equation}\label{eq9}
  |\mathcal G_0|<\frac 1{13}2^{w}.
\end{equation}
Define $T_0:=[w+1,w+7]\cup \{w+9,w+10,w+12,w+13,\ldots\}\cap [n]$.
Now we can define
$$\mathcal G_1:=\{G\subset [w]: G\notin \mathcal G_0, T_0\notin \mathcal F(G,[w])\}.$$
Here we invoke an old result of the first author \cite[Lemma 2]{F11} which asserts that for any $G\in \mathcal G_1$
\begin{equation}\label{eq10}
  |\mathcal F(G,[w])|<\Big(\frac{\sqrt 5-1}2\Big)^8 2^{n-w}<\frac 1{46} 2^{n-w}.
\end{equation}
Finally, set $\mathcal G_2:=2^{[w]}\setminus (\mathcal G_0\cup \mathcal G_1)$. By construction, the $37$-element set $\tilde H$ is in $\mathcal G_2$. Below we are going to prove the following.
\begin{prop}\label{prop13}
$\mathcal G_2$ is $3$-wise intersecting.
\end{prop}
Let us first show how Proposition~\ref{prop13} implies $|\ff|<2^{n-2}$. First note that the pairwise $38$-intersecting property and $|G|\le 37$ for all $G\in \mathcal G_2$ imply that for any $G\in \mathcal G_2$ the family $\mathcal F(G,[w])$ is $2$-wise intersecting. Consequently, $|\mathcal F(G,[w])|\le \frac 12 2^{n-w}$.

Partition $\ff$ according to $F\cap [w]$: $\ff_i:=\{F\in\ff:F\cap [w]\in \mathcal G_i\}$. We have \begin{equation}\label{eq11} |\ff|=|\ff_0|+|\ff_1|+|\ff_2|=|\mathcal G_0|\cdot 2^{n-w}+|\mathcal G_1|\cdot \frac 1{46} 2^{n-w}+|\mathcal G_2|\cdot \frac 12 2^{n-w}.
\end{equation}
By $\tilde H\in\mathcal G_2$ and Proposition~\ref{prop13}, we may apply the Brace--Daykin Theorem and infer
\begin{equation}\label{eq12}
  |\mathcal G_2|\le \frac 5{16} 2^{n-w}.
\end{equation}
Since the coefficient in front of $|\mathcal G_1|$ is the smallest, we get an upper bound for the RHS for \eqref{eq11} by making $|\mathcal G_0|=\frac 1{13} 2^{w}$, $|\mathcal G_2|=\frac 5{16} 2^{w}$ and $|\mathcal G_1|=\big(1-\frac 1{13}-\frac 5{16}\big)2^w$. We obtain
$$|\ff|\le \Big(\frac 1{13}+\frac 5{32}+\big(1-\frac 1{13}-\frac 5{16}\big)\frac 1{46}\Big)2^n<2^{n-2},$$
as desired.

\begin{proof}[Proof of Proposition~\ref{prop13}]
Take first $F,G,H\in \mathcal G_2\setminus \{H_0\}$ and suppose that $F\cap G\cap H=\emptyset$. By definition, $T_0\in \mathcal F(S,[w])$ for $S=F,G$ and $H$. Using shiftedness, we can obtain that \begin{align*}T':=[w+1,w+7]\cup \{w+8,w+10,w+11,w+13,\ldots\}&\in \mathcal F(G,[w])\ \ \ \text{and}\\
T'':=[w+1,w+7]\cup \{w+8,w+9,w+11,w+12,\ldots\}&\in \mathcal F(H,[w]).
\end{align*}
The intersection of $T',T''$ and $T_0$ is $[w+1,w+7]$. Since $|F|+|G|+|H|\le 3\cdot 33=99<2\cdot 54-7$, for each $i\in[7]$ we can replace $w+i$ with an element $[w]$ in one of $F\cup T_0, G\cup T', H\cup T''$ and strictly decrease the common intersection of the three sets. Repeating it for each $i\in[7]$, by shiftedness we get that there are three sets in $\ff$ that have empty common intersection, a contradiction.

Now suppose that $H=\tilde H=H_0\cap [w]$. Then $\mathcal F(H,[w])$ contains $H':=H\cap [w+1,n]=\{w+1,w+2,w+4,w+5,w+7,w+8,\ldots\}$. Taking $T_0\in \mathcal F(F,[w])$ and $T''\in \mathcal F(G,[w])$, respectively, we get that $T''\cap T_0\cap H'=\{w+1,w+2,w+4,w+5,w+7\}$. To arrive at the same contradiction, we shift these $5$ elements into $[w]$, decreasing the intersection of $T_0\cup F, T''\cup G$ and $H_0$ after each shift. Since $|F|+|G|+|\tilde H|\le 33+33+37=2\cdot 54-5,$ this is possible.
\end{proof}

\textbf{Case 2. $H_0\notin\ff$. } This condition implies that, for all $S\subset[2,7]$ and $F\in \ff(S,[7])$,  there exists $\ell$ such that
\begin{equation}\label{eqinters}
                             |F\cap [8,3\ell+9]|\ge 2\ell+2.
\end{equation}
Indeed, it is true for $S=[2,7]$ since $H_0\cap [8,n]$ is the unique maximal set in the shifting/inclusion order that does not have this property, and for $S'\subset S$ we have $\ff(S,[7])\supset \ff(S',[7])$. The equations \eqref{eqinters} and \eqref{eqnew6}, in turn, imply that, for each $S\subset [2,7]$, we have $|\partial(\ff(S,[7]))|\ge 2 |\ff(S,[7])|$.

For a two-element set $\{x_i,y_i\}$, let us consider the following four ordered triplets:
\begin{alignat*}{4}
  &(\emptyset, && \{x_i\}, &&\{x_i,y_i\}&&), \\
  &(\{x_i\}, && \{y_i\}, &&\{y_i\}&&), \\
  &(\{y_i\}, && \{x_i,y_i\},\ && \emptyset&&), \\
  &(\{x_i,y_i\},\ &&\emptyset, &&\{x_i\}&&).
\end{alignat*}
Note that all four subsets of $\{x_i,y_i\}$ occur once in each position (column). Also, the sum of  sizes of the subsets in each triplet is always $3$ and the intersection of the subsets is empty. Suppose that $\{x_1,x_2,x_3,y_1,y_2,y_3\}=[2,7]$ and let $(A_i,B_i,C_i)$, $i\in [3]$, be some of the above triples. We associate with them a {\it big triple}
$$(\{1\}\cup A_1\cup A_2\cup A_3, \{1\}\cup B_1\cup B_2\cup B_3,C_1\cup C_2\cup C_3).$$
Let us note that, for each big triple, the sum of the sizes of the subsets in it is $11$. Altogether, we constructed $4\times 4\times 4=64$ triples, where each subset of $[7]$ containing $1$ appears exactly once in the first and second position and each subset of $[2,7]$ appears exactly once in the third position. Moreover, the intersection of the three subsets is empty for each triple.

For a big triple $(A,B,C)$ we consider the three families $\ff(D):=\ff(D,[7])$, where $D=A,B,$ or $C$. Recall that $\sigma (\ff) = \ff\cup \partial \ff$.
\begin{prop}
  The families $\ff(A),\ \ff(B),\ \ff(C)$ are cross $3$-wise $4$-intersecting. The families $\sigma(\ff(A)),\ \sigma(\ff(B)),\ \sigma(\ff(C))$ are cross $3$-wise intersecting.
\end{prop}
\begin{proof}
    For each triple $(A,B,C)$, either there are three elements in $[2,7]$ that are contained in  only one set among $A,B,C$, or one such element and one element which is not contained in $A\cup B\cup C$. In either case, if $F\in \ff(A)$, $G\in \ff(B)$, $H\in \ff(C)$ satisfy $|F\cap G\cap H|\le 3$, then we can do (at most) three shifts and replace each element that belongs to the intersection in one set with one of the ``low-degree'' elements, thus not creating new common intersection. By shiftedness, we will get $F', G', H'$ that belong to $\ff$ but whose common intersection is empty.

  The second statement obviously follows from the first one.
\end{proof}

Now, if $\ff(D)$ is non-empty then $\sigma(\ff(D)$ is non-trivial, where $D=A,B,C$. In that case, by \eqref{eqnew7}
\begin{equation}\label{ineqfra}
  |\sigma(\ff(A)|+|\sigma(\ff(B))|+|\sigma(\ff(C)|\le 2^{n-7}.
\end{equation}
 On the other hand, if one of the families above is empty, the sum of cardinalities of the two remaining ones is at most $2^{n-7}$ since they are cross-intersecting (due to the $2$-wise $38$-intersecting property). Note that $1\notin C$ implies that $\ff(C)$ is $3$-wise $2$-intersecting. In view of Corollary~\ref{cor10}, we infer $|\sigma(\ff(C))|\ge 2|\ff(C)|$. Consequently,  in all cases we have
\begin{equation*}
  |\ff(A)|+|\ff(B)|+2|\ff(C)|\le 2^{n-7}.
\end{equation*}

Summing over the $64$ big triples gives $2|\ff| = 2\sum_{D\subset [7]}|\ff(D)|\le 64 \cdot 2^{n-7}$, that is, $|\ff|\le 2^{n-2}$.


\begin{thebibliography}{100}
\bibitem{BD} A. Brace and D. E. Daykin, {\it A finite set covering theorem}, Bulletin of the Australian Mathematical Society 5 (1971), N2, 197--202.

\bibitem{EKR} P. Erd\H os, C. Ko and R. Rado, \textit{Intersection theorems for systems of finite sets}, The Quarterly Journal of Mathematics, 12 (1961) N1, 313--320.

\bibitem{F11} P. Frankl, {\it Families of finite sets satisfying an intersection condition}, Bull. Austral. Math. Soc. 15 (1976), N1, 73--79.

\bibitem{Fra3} P. Frankl,  \textit{The shifting technique in extremal set theory}, Surveys in combinatorics 123 (1987), 81--110.
\bibitem{F5} P. Frankl, {\it Shadows and shifting}, Graphs and Combinatorics 7 (1991), 23--29.

\bibitem{F10} P. Frankl, {\it Some exact results for multiply intersecting families}, to appear
\bibitem{Kat64} G.O.H. Katona, {\it Intersection theorems for systems of finite sets}, Acta Math.
Acad. Sci. Hungar. 15 (1964), 329--337.

\bibitem{Kl66}  D.J. Kleitman, {\it Families of Non-Disjoint Subsets}, J. Combin. Theory 1 (1966), 153--155.

\end{thebibliography}
\end{document}